\documentclass[a4paper,12pt]{article}

\usepackage{amsfonts,amsthm,amssymb,amsmath}
\usepackage{hyperref}
\usepackage{fourier}
\usepackage[utf8]{inputenc}
\usepackage{color}

\newtheorem{theorem}{Theorem}[section]

\newtheorem{corollary}[theorem]{Corollary}
\newtheorem{lemma}[theorem]{Lemma}
\theoremstyle{definition}

\newtheorem{remark}[theorem]{Remark}

\newcommand{\N}{\mathbb{N}}					
\newcommand{\T}{\mathbb{T}}

\usepackage{mathtools}

\DeclarePairedDelimiter{\abs}{\lvert}{\rvert}
\DeclarePairedDelimiter{\norm}{\lVert}{\rVert}

\DeclareMathOperator{\mon}{mon}

\newcommand{\mmn}{\mathcal M(m,n)}

\newcommand{\jmn}{\mathcal J(m,n)}

\newcommand{\jbn}[1]{\mathcal J(#1,n)}
\newcommand{\bi}{\mathbf i}
 \newcommand{\bj}{\mathbf j}

\newcommand{\veps}{\varepsilon}
\newcommand{\chimon}{\chi_{\mon}}
\newcommand{\pjlr}{\mathcal P(^{J}\ell_r)}
\newcommand{\fullpjlr}{\mathcal P(^{\mathcal J(m,n)}\ell_r)}

\title{Monomial convergence for holomorphic functions
on $\ell_r$}
\author{
Fr\'ed\'eric Bayart \footnote{Laboratoire de Math\'ematiques
Universit\'e Blaise Pascal Campus des C\'ezeaux, F-63177 Aubi\`ere Cedex (France)
}
\and Andreas Defant\footnote{Institut f\"ur Mathematik. Universit\"at Oldenburg. D-26111 Oldenburg (Germany) }
\and Sunke Schl\"uters\footnote{Institut f\"ur Mathematik. Universit\"at Oldenburg. D-26111 Oldenburg (Germany) }
}
\date{}

\begin{document}
\maketitle

\begin{abstract}
	\noindent
Let $\mathcal F$ be either the set of all bounded holomorphic functions or the set of all $m$-homogeneous polynomials on the unit ball of $\ell_r$. We give a systematic study of the sets of all $u\in\ell_r$ for which the monomial expansion $\sum_{\alpha}\frac{\partial^\alpha f(0)}{\alpha !}u^\alpha$ of every $f\in\mathcal F$ converges. Inspired by recent results from the general theory of Dirichlet series, we establish as our main tool,  independently  interesting, upper estimates for the unconditional basis constants of spaces of polynomials on $\ell_r$ spanned by finite sets of monomials.
\end{abstract}



\section{Introduction}
Let $X$ be a Banach sequence space (i.e., $\ell_1 \subset X \subset c_0$ such that the canonical sequences $(e_k)$ form a $1$-unconditional basis) and $R \subset X$ a Reinhardt domain (i.e., a nonempty open set
such that any complex sequence $u$ belongs to $R$  whenever  there exists $z\in R$ with $|u|\leq |z|$; for instance, the open unit ball $B_X$ of $X$). Then each holomorphic (i.e., Fr\'{e}chet differentiable)  function $f:R \rightarrow \mathbb{C}$ has a power series expansion $\sum_{\alpha\in\mathbb N_0^n}c_\alpha^{(n)}z^\alpha$ on every finite dimensional section $R_n$ of $R$, and for example from the Cauchy formula we can
 see that $c_\alpha^{(n)}=c_\alpha^{(n+1)}$ for $\alpha\in\mathbb N_0^n\subset \mathbb N_0^{n+1}$. Thus there is a unique family $(c_\alpha(f))_{\alpha\in\mathbb N_0^{(\mathbb N)}}$ such that,
 for  all $n\in\mathbb N$ and all $z\in R_n$,
$$f(z)=\sum_{\alpha\in\mathbb N_0^{(\mathbb N)}}c_\alpha z^\alpha.$$
The power series $\sum_\alpha c_\alpha z^\alpha$ is called the monomial expansion of $f$, and $c_\alpha=c_\alpha(f)$ are its monomial coefficients.

Contrary to what happens on finite dimensional domains, the monomial expansion of $f$ does not necessarily converge at every point of $R$.
This in \cite{DeMaPr09} motivated the introduction of the following definition: Given a subset $\mathcal F(R)$  of $H(R)$, the set of all holomorphic functions on $R$, we call
\begin{equation*}
    \mon \mathcal F(R)=\Bigg\{z\in R\,\,\,  : \,\,\, \sum_{\alpha\in\mathbb N_0^{(\mathbb N)} } \big|c_\alpha(f) z^\alpha \big| < \infty \,\, \text{ for all } f \in \mathcal F(R)  \Bigg\}
\end{equation*}
the domain of monomial convergence with respect to $\mathcal F(R)$.

By continuity of a holomorphic function, and since the equality is satisfied on $R_n$, we know that for all $z\in\mon\mathcal F(R)$,
$$f(z)= \sum_{\alpha\in\mathbb N_0^{(\mathbb N)} } c_\alpha(f) z^\alpha.$$
We are mostly interested in determining $\mon \mathcal F(R)$ when $\mathcal F(R)=\mathcal P(^m\ell_r)$ or $ H_\infty(B_{\ell_r})$ for $1 \leq r \leq \infty$; as usual we denote by $H_\infty(B_X)$ the Banach space of all bounded holomorphic functions $f: B_X \rightarrow \mathbb{C}$, and by $\mathcal{P}(^mX)$ its closed subspace of all $m$-homogeneous polynomials $P$ (i.e., all restrictions of bounded $m$-linear forms on $X^m$ to their diagonals).

The case $r=1$ was solved completely by Lempert in \cite{Le99}, and the case $r=\infty$ seems fairly well-understood through the results of \cite{BaDeFrMaSe14} (for more on these results see the introductions of  the sections \ref{polynomials}  and \ref{holomorphic functions}). However, for $1< r<\infty$, despite the results of \cite{DeMaPr09}, the description of $\mon\mathcal P(^m\ell_r)$ and $\mon H_\infty(B_{\ell_r})$ remains mysterious.  In this paper, we improve the  knowledge on these cases.

 Of course, for $X=\ell_1$   the fact that each sequence in $\ell_1$ by definition is absolutely summable is a big
 advantage, and for $X=\ell_\infty$ the crucial  tool is the  Bohnenblust-Hille inequality (an inequality for $m$-linear forms on $\ell_\infty$) together with all its recent improvments.
 But for
$X=\ell_r$ with $r \neq 1, \infty$ we need alternative techniques.

The problem is to find for each  $u \in B_{\ell_r}$ an additional {\it summability condition} which guaranties full control of all sums
$
\sum_{\alpha } \left|c_\alpha(f) u^\alpha\right|\,,\, f \in H_\infty(B_{\ell_r})\,.
$
The general idea is simple. Split   the  set $\mathbb N_0^{(\mathbb N)}$  of all multi indices $\alpha$ into  a union of finite sets $\Lambda_n$, and  then  each $\Lambda_n$  into  the disjoint union of all its $m$-homogeneous parts $\Lambda_{n,m}$ (i.e.,
 all $\alpha \in \Lambda_n$ with  order $|\alpha|=m$). The  challenge now is as follows: Find a clever decomposition
 \begin{equation} \label{deco}
\mathbb N_0^{(\mathbb N)} = \bigcup_{m,n} \Lambda_{m,n}\,,
\end{equation}
  which allows   a in a sense  uniform control over all possible partial sums
 \begin{equation} \label{smallsums}
\sum_{\alpha \in \Lambda_{n,m}} \left|c_\alpha(f) u^\alpha\right|\,\,, \quad f \in H_\infty(B_{\ell_r})\,,
\end{equation}
such that under the additional  summability property of $u\in B_{\ell_r}$ we for all functions $f$ finally can conclude that
\[
\sum_{\alpha \in \mathbb N_0^{(\mathbb N)} } \left|c_\alpha(f) u^\alpha\right|
\leq
\sum_n\,\, \sum_{m} \,\sum_{\alpha \in \Lambda_{n,m}} \left|c_\alpha(f) u^\alpha\right| < \infty\,.
\]
In order to study domains $ \mon \mathcal F(R)$ of monomial convergence, the  decomposition  in \eqref{deco} which for our purposes is crucial, is inspired by the work of  Konyagin and Queff\'elec  from \cite{KONQUEFF} on  Dirichlet series (see \ref{beginning}),
and it is
 based  on the fundamental theorem of arithmetics.
In order to handle \eqref{smallsums}, we study  for arbitrary finite index sets $\Lambda$ of multi indices
 upper  bounds
of the unconditional basis constant of the subspace in $\mathcal{P}(^m\ell_r)$  spanned by  all monomials
$z^\alpha\,, \alpha \in \Lambda$.
 Two   tools of seemingly independent interest are established.
The first one
is a fairly general  upper  estimate whenever all $\alpha \in \Lambda$ are $m$-homo\-geneous (i.e., $|\alpha| =m$) (Theorem \ref{THMMONOMIAL}).
The  second one leads to such estimates for certain sets $\Lambda$ of  nonhomogeneous $\alpha's$, needed  to apply the above
technique of Konyagin and Queffélec (Theorem~\ref{unc-basis-q} and \ref{unc-basis-pol}).
  Finally, we present  our  new results  on  sets of monomial convergence
for homogeneous   polynomials and bounded  holomorphic functions on $\ell_r$ (for polynomials see part (3),(4) of  Theorem
\ref{zero} and  Theorem \ref{LEMPRESQUEPOLY}, and for holomorphic functions Theorem \ref{thm:psigma} with its corollaries \ref{CORHOL1} and \ref{CORHOL}).

\section{Preliminaries}
We use standard notation from Banach space theory. As usual,  we denote the conjugate exponent of $1 \le r \le \infty$ by $r'$, i.e. $\tfrac{1}{r} + \tfrac{1}{r'} = 1$. Given $m,n \in \mathbb{N}$ we consider the following sets of indices
\begin{align*}
    \mathcal{M} (m,n)
        &=  \big
                \{\mathbf{j} = (j_{1}, \dots , j_{m})
            \,;\,
                1 \leq j_{1}, \dots , j_{m} \leq n
            \big\}
            = \{1, \ldots , n \}^{m} \\
    \mathcal M(m)
        &= \mathbb N^m \\
    \mathcal M
        &= \mathbb N^\mathbb{N}
\intertext{and}
    \mathcal{J} (m,n)
        &= \big\{
                \mathbf{j} \in \mathcal{M} (m,n)
            \,;\,
                1 \leq j_{1} \leq \dots \leq j_{m} \leq n
            \big\} \\
    \mathcal{J}(m)
        &= \bigcup_{n}\mathcal J(m,n) \\
    \mathcal{J}
        &=\bigcup_{m}\mathcal J(m).
\end{align*}
For indices $\bi,\bj \in \mathcal{M}$ we denote by $(\bi,\bj) = (i_1, i_2, \dotsc, j_1, j_2, \dotsc)$ the concatenation of $\bi$ and $\bj$. An equivalence relation is defined in $\mathcal{M} (m)$ as follows: $\mathbf{i} \sim \mathbf{j}$
if there is a permutation $\sigma$ such that $i_{\sigma
(k)} = j_{k}$
for all $k$.
We write $|\mathbf{i}|$ for the cardinality of the equivalence class
$[\mathbf{i}]$.
Moreover, we note that for each $\mathbf{i} \in \mathcal{M} (m)$ there is a unique $\mathbf{j} \in \mathcal{J} (m)$
such that $\mathbf{i} \sim \mathbf{j}$.

Let us compare our index notation with the  multi index notation usually used in the context of polynomials.
There is a one-to-one relation between  $\mathcal{J} (m)$
and
\[
\Lambda (m) = \left\{ \alpha \in \mathbb{N}_{0}^{(\mathbb N)}  \,\,;\,\, \vert \alpha \vert = \sum_{i=1}^{\infty}\alpha_i =m\right \}\,;
\]
indeed, given $\mathbf{j}$,
one can define $\alpha$ by doing $\alpha_{r} = | \{ q \,|\, j_{q}=r \}|$;
conversely, for each $\alpha$, we consider
$\mathbf{j}_{\alpha} = (1, \stackrel{\alpha_{1}}{\dots} , 1, 2,\stackrel{\alpha_{2}}{\dots} ,2 ,
\dots , n ,\stackrel{\alpha_{n}}{\dots} ,n,\dots)$. In the same way we may identify
$\Lambda (m,n) = \left\{ \alpha \in \mathbb{N}_{0}^{n}  \,\,;\,\, \vert \alpha \vert = m\right \}$
with $\mathcal{J} (m,n)$.
Note that $|\mathbf{j}_{\alpha}| = \frac{m!}{\alpha !}$ for every $\alpha \in \Lambda (m)$.
Taking this correspondence into account, for every Banach sequence space $X$ the monomial series expansion of a $m$-homogeneous polynomial $P \in \mathcal{P} (^{m} X)$
can be expressed in
different ways
(we write $c_{\alpha} = c_{\alpha}(P)$)
\begin{equation*} \label{polin varios}
  \sum_{\alpha \in \Lambda(m)} c_{\alpha} z^{\alpha} = \sum_{\mathbf{j} \in \mathcal{J}(m)} c_{\mathbf{j}} z_{\mathbf{j}}
= \sum_{1 \leq j_{1} \leq \ldots \leq j_{m}} c_{j_{1} \ldots j_{m}} z_{j_{1}} \cdots z_{j_{m}} \, .
\end{equation*}
Given a Banach sequence space $X$ and some index subset   $J \subset \mathcal{J}$, we write $\mathcal P(^J X)$ for the closed subspace of all  holomorphic functions $f \in H_\infty(B_X)$ for which $c_\bj(f) =0$ for all $\bj \in \mathcal{J} \setminus J$. Clearly, $\mathcal P(^m X)= \mathcal P(^{\mathcal{J}(m)} X)$. If $J \subset \mathcal{J}$ is finite, then
\begin{equation*}
    \mathcal P(^J X)=\textrm{span}\Big \{z_\bj\,: \,\ \bj\in J\Big\}\,,
\end{equation*}
where  $z_\bj$ for $ \bj=(j_1, \ldots, j_\ell)$ stands for the monomial $z_\bj: u \mapsto u_\bj:=u_{j_1}\cdot \ldots \cdot u_{j_\ell}$. For $J\subset\mathcal J(m)$, we call
\begin{equation*}
    J^*=\big\{\bj \in\mathcal J(m-1);\ \exists k\geq 1,\ (\bj,k)\in J\big\}
\end{equation*}
the reduced set of $J$.

\section{Unconditionality} \label{uncon}
Given a compact group $G$, the Sidon constant of a finite set $\mathcal{C}$ of characters $\gamma$ (in
the dual group)  is the best constant $c \ge 0$, denoted by $S(\mathcal{C})$, such that for every choice of scalars $c_\gamma,
 \gamma  \in \mathcal{C}$, we have that
\[
\sum_{\gamma  \in \mathcal{C}}  |c_\gamma| \leq \,c\,\Big\| \sum_{\gamma \in \mathcal{C}}  c_\gamma \gamma\Big\|_\infty\,.
\]
An immediate consequence of the Cauchy-Schwarz inequality is  that
\[
1 \leq S(\mathcal{C}) \leq |\mathcal{C}|^{\frac{1}{2}}\,.
\]
For the  circle groups $G=\mathbb{T}, \mathbb{T}^n$ and $\mathbb{T}^\infty$ different values are possible:
\begin{itemize}
\item
A well-known result of Rudin shows that for the set $\mathcal C=\{1,z,\ldots,z^{n-1}\}$ of characters on $G=\mathbb{T}$
we have, up to constants independent of $n$,
\begin{equation} \label{num1}
S(\mathcal{ C}) \asymp \sqrt{n}\,.
\end{equation}
\item
In \cite{DFOOS} it was proved  that
for every $m,n$ the Sidon constant of the monomials $\mathcal C=\{z^\alpha : \alpha \in  \Lambda(m,n)\}$  on
$G=\mathbb{T}^n$,  up to the $m$th power $C^m$ of some absolute constant $C$, satisfies
\begin{equation} \label{num2}
S(\mathcal{ C}) \asymp \left|\Lambda(m-1,n)\right|^{\frac{1}{2}}\,.
\end{equation}
\item
In contrast, a reformulation of a result of Aron and Globevnik \cite[Thm 1.3]{AronGlobevnik} shows that
for every $m$ the Sidon constant of the sparse set   $\mathcal C=\{z_j^m : j \in \mathbb{N}\}$
fulfills
\begin{equation} \label{num3}
S(\mathcal{ C}) =1 \,.
\end{equation}
\end{itemize}
Let us transfer some of these results into terms of unconditional bases constants of spaces polynomials on sequence spaces.
Recall that a Schauder basis $(x_n)$ of a Banach space $X$ is said to be unconditional whenever there is a constant $c\geq 0$
such that
    $\left\| \sum_{k} \veps_k \alpha_k x_k \right\|
        \leq c \left\| \sum_k \alpha_k x_k\right\|$
for every $x = \sum_k a_k x_k \in X$ and all choices of $( \veps_k)_k \subset \mathbb C$ with $|\veps_k|=1$. In this case, the best constant $c$ is denoted by $\chi\big((x_n)\big)$ and called the
unconditional basis constant of $(x_n)$. If such a constant doesn't exist, i.e. if the basis is not unconditional, we set $\chi\big((x_n)\big) = \infty$.

Given a Banach sequence space $X$ and an index set  $J \subset \mathcal{J}$, such that the set $\mathcal{C}=\{z_\bj\,:\, \bj \in J\}$ of all monomials associated with $J$ (counted in a suitable way) forms
an   basis of $\mathcal{P}(^J X)$, we write
\begin{equation*}
    \chimon\big(\mathcal{P}(^J X)\big) = \chi( \mathcal{C} ) \,.
\end{equation*}
If we interpret each of these monomials $z_\bj$ as a  character on the group $\mathbb{T}^\infty$,
then a straightforward calculation (using the distinguished maximum modulus principle) proves that
$$
S(\mathcal{C}) = \chimon\big(\mathcal{P}(^J \ell_\infty)\big)\,.
$$
A simple but useful lemma shows that  $\chimon\big(\mathcal{P}(^J \ell_\infty)\big)$ is an upper bound of all  $\chimon \big(\mathcal{P}(^J X)\big)$.
\begin{lemma}
    \label{lem:the_trick_poly}
    Let $X$ be a Banach sequence space and let $J\subset\mathcal J$, such that the monomials form a basis of $\mathcal P(^J X)$. Then
    \begin{equation*}
        \chimon\big(\mathcal P(^J X)\big)\leq \chimon\big(\mathcal P(^J \ell_\infty)\big).
    \end{equation*}
\end{lemma}
\begin{proof}
    Assume $\chimon\big( \mathcal P(^J \ell_\infty) \big) < \infty$ (otherwise there is nothing to show). For $P\in \mathcal P(^J X)$ and a fixed  $u\in B_X$ define $Q(w)=P(wu)\in\mathcal P(^J \ell_\infty)$.
    Since $B_X$ is a Reinhardt domain, we have
    $\|Q\|_\infty\leq \|P\|_\infty$. It is now sufficient to observe that
    \begin{align*}
        \sum_{\bj\in J} |c_\bj(P) u_\bj |
        &=
        \sup_{w\in B_{\ell_\infty}}\sum_{\bj\in J}|c_\bj(P) u_{\bj}||w_{\bj}|
        = \sup_{w\in B_{\ell_\infty}}\sum_{\bj\in J}|c_\bj(Q)||w_{\bj}|
        \\
        &\leq \chimon\big(\mathcal P(^J \ell_\infty)\big) \|Q\|_\infty
        \leq \chimon\big(\mathcal P(^J \ell_\infty)\big) \|P\|_\infty\,,
    \end{align*}
    the conclusion.
\end{proof}

\noindent
Let us again see some examples:
Given $X$, an immediate consequence of \eqref{num1} is that for $\mathcal P(^J X\big) = \operatorname{span} \{ z_1^j \,;\, 0 \leq j \leq n-1 \}$ we have, up to a universal constant,
\begin{equation*}
    \chimon\big(\mathcal P(^J X)\big) \asymp \sqrt{n},
\end{equation*}
and from  \eqref{num3} we may deduce that for $J=\big\{ (k, \dotsb, k) \,;\, k\in\mathbb N \big\} \subset \mathcal{J}(m) $
\[
\chimon\big(\mathcal P(^J X)\big) =1.
\]
Generalizing \eqref{num2}  is much more complicated. In the scale of  all $\ell_r$-spaces
the results from \cite{BAYMAXMOD} (lower estimates) and \cite{ DeFr11, DFOOS} (upper estimates) show that for $1 \leq r \leq \infty$
\begin{equation}
    \label{start}
    \chimon\big(\fullpjlr\big)   \asymp \big|\mathcal{J}(m-1,n)\big|^{1-\frac{1}{\min{(r,2)}}}\,,
\end{equation}
where $\asymp$ means that the left and the right side equal up to the $m$-th power $C^m$ of a constant only depending on $r$ (and neither on $m$ nor on $n$).

Replacing the index set $\mathcal{J}(m,n)$ by an arbitrary finite subset $J\subset \mathcal J(m,n)$ the following result is a strong improvement
and our main tool within our later study of sets  of monomial convergence.

\begin{theorem}\label{THMMONOMIAL}
    Given  $1 \le r \le \infty$ and $m \ge 1$, there is a constant $C(m,r) \ge 1$ such that for every $n\geq 1$, every $P \in \mathcal{P}(^{\mathcal{J}(m,n)} \ell_r)$,  every
     $J\subset \mathcal J(m,n)$, and every
    $u \in \ell_r$ we have
    \begin{equation} \label{REMMONOMIAL}
        \sum_{\bj\in J} \left|c_{\bj}(P)\right||u_\bj|\leq C(m,r) |J^*|^{1-\frac 1{\min(r,2)}}\|u\|_r^m\|P\|_\infty\,,
    \end{equation}
    where
    \begin{equation*}
        C(m,r) \leq \begin{cases}
            e me^{(m-1)/r} & \text{if $1 \leq r \leq 2$} \\
            e m2^{(m-1)/2} & \text{if $2\leq r \leq \infty$.}
        \end{cases}
    \end{equation*}
    In particular, for every finite $J \subset \mathcal{J}(m)$
    \begin{equation} \label{REMMONOMIAL1}
        \chimon\big(\pjlr\big)\leq  C(m,r) |J^*|^{1-\frac 1{\min(r,2)}}\,.
    \end{equation}
\end{theorem}

\noindent
The proof is given in the following two subsections; it is different for $r\leq 2$ and for $r\geq 2$.
The case  $r=\infty$ of \eqref{start} is given in  \cite{DFOOS}, and it  uses the hypercontractive Bohnenblust-Hille inequality. The general case  $1 \leq r \leq \infty$ from \cite{DeFr11} needs sophisticated  tools  from local Banach space theory (as Gordon-Lewis and projection constants). These arguments in fact only work for the whole index set $\mathcal{J}(m,n)$, and they seem to fail in full generality for subsets $J$ of $\mathcal{J}(m,n)$.  We here provide a tricky, but quite elementary, argument which works for  arbitrary  $J$;
moreover, we  point out   that even for the special case $J= \mathcal{J}(m,n)$ we obtain  better  constants $C(m,r)$ for  \eqref{REMMONOMIAL1} than in  \cite{DeFr11}.

From \cite{DK05} we know that for each infinite dimensional Banach sequence space $X$, the Banach space $\mathcal{P}(^mX)$ never has an unconditional basis. In particular, the unconditional basis constant
$\chimon \big(\mathcal{P}(^mX)\big)$ of all monomials $(z^\bj)_{\bj \in \mathcal{J}(m)}$ is not finite. But let us note that in contrast to this there are $X$ such that for each $m$
\begin{equation*}
    \sup_n \chimon \big(\mathcal{P}(^{\mathcal{J}(m,n)}X) \big)< \infty
\end{equation*}
(this can be easily shown for $X =\ell_1$, but following \cite{DK05} there are even examples of this type different from  $\ell_1$).

\subsection{The case $ r\leq 2$}
We need several lemmas. The first one is a Cauchy estimate and can be found in  \cite[p. 323]{Boas00}. For the sake of completeness we include a streamlined  argument.

\begin{lemma}\label{coef}
    Let $1 \leq r \leq \infty$ and $\alpha \in \mathbb{N}_0^n$ with $|\alpha|=m$\,.
    Then for each $ P\in \mathcal{P}\big(^{m} \ell_r^n\big)$ we have
    \begin{equation*}
        \abs[]{ c_\alpha(P) }
        \leq \Big( \frac{m^m}{\alpha^\alpha} \Big)^{\frac{1}{r}}
            \norm[]{ P }_\infty \,.
    \end{equation*}
    In particular, for each $\bj \in\mathcal J(m,n) $ we have that
    \begin{equation*}
        \abs[]{ c_\bj (P) }
        \leq e^{\frac{m}{r}} \abs[]{ \bj }^{\frac{1}{r}}
            \norm{P}_\infty \,.
    \end{equation*}
\end{lemma}

\begin{proof}
Define $u= m^{-1/r} (\alpha_1^{1/r},  \ldots,\alpha_n^{1/r}) \in B_{\ell_r^n}$.
 Then by the Cauchy integral formula for each $P \in \mathcal{P}\big(^{m} \ell_r^n\big) $
\[
c_\alpha(P) = \frac{1}{(2\pi i)^n} \int_{|z_n|=u_n} \ldots \int_{|z_1|=u_1}
\frac{P(z)}{z^\alpha z_1 \ldots z_n} dz\,.
\]
Hence we obtain
\[
|c_\alpha(P)| \leq \frac{1}{|u^\alpha|} \|P\|_\infty= \Big( \frac{m^m}{\alpha^\alpha} \Big)^{\frac{1}{r}}\|P\|_\infty\,,
\]
the conclusion. For the second inequality note first that $ \Big( \frac{m^m}{\alpha^\alpha} \Big)^{\frac{1}{r}}\leq e^{m/r}  \Big(\frac{m!}{\alpha!}\Big)^{1/r}\,,$ and recall that if we associate to $\bj$ the multi index $\alpha$, then
$\frac{m!}{\alpha!} = |\bj|$\,.
\end{proof}

\begin{corollary}\label{LEMPOLYMIXTE}
 Consider the linear operator $Q \in \mathcal L\big( \ell_r^n,\mathcal P(^{m-1}\ell_r^n)\big)$ defined by
\begin{equation*}
  Q(z,w)=\sum_{\bj\in\jbn{m-1}}\left(\sum_{k=1}^n b_{(\bj,k)}z_{\bj}\right)w_k
  \, ,
\end{equation*}
where $z,w \in \ell_r^n$. Then for any $\bj\in\jbn{m-1}$ ,
$$\left(\sum_{k=1}^n |b_{(\bj,k)}|^{r'}\right)^{1/r'}\leq e^{\frac{m-1}r}{ |\bj|}^{1/r} \|Q\|_\infty.$$
\end{corollary}
\begin{proof}
 Let us fix $w\in B_{\ell_r^n}$. Then $Q(\cdot,w)\in \mathcal P(^{m-1}\ell_r^n)$. Thus, by the preceding lemma for any $\bj\in\jbn{m-1}$,
\begin{eqnarray*}
 \left|\sum_{k=1}^n b_{(\bj,k)}w_k\right|\leq e^{\frac{m-1}r}{ |\bj|}^{1/r}\sup_{z\in B_{\ell_r^n}}|Q(z,w)|
\leq&e^{\frac{m-1}r}{ |\bj|}^{1/r}\|Q\|_\infty.
\end{eqnarray*}
We now take the supremum over all possible $w\in B_{\ell_r^n}$.
\end{proof}

\begin{lemma}\label{LEMPOLY}
 Let $P\in\mathcal P(^m\ell_r^n)$. Then for any $\bj\in \jbn{m-1}$
$$\left(\sum_{k=j_{m-1}}^n |c_{(\bj,k)}(P)|^{r'}\right)^{1/r'}\leq m e^{1+\frac{m-1}r}{|\bj|}^{1/r} \|P\|_\infty.$$
\end{lemma}
\begin{proof}
 Let $A:\ell_r^n \times \ldots \times \ell_r^n \rightarrow \mathbb{C} $ be the symmetric $m$-linear form associated to $P$,
$$A(z^{(1)}, \ldots,z^{(m)} )=\sum_{\bi\in\mmn}a_\bi(A) z_{i_1}^{(1)}\cdots z_{i_m}^{(m)}\,;$$
in particular, for each $\bj\in \jbn{m}$  we have
$
a_{\bj}(A)  =\frac{c_{\bj}(P)} {|\bj|}.
$
For  $z,w \in \ell_r^n$  define the linear operator
$$Q(z,w)= A(z,\dots,z,w)\in\mathcal L\left(\ell_r^n,\mathcal P(^{m-1}\ell_r^n)\right)\,;$$
then a simple calculation proves
\begin{eqnarray*}
 Q(z,w)&=&\sum_{\bi\in\mmn}a_\bi(A) z_{i_1}\cdots z_{i_{m-1}} w_{i_m}\\
&=&\sum_{\bi\in \mathcal{M}(m-1,n)} \sum_{k=1}^n a_{(\bi,k)}(A) z_{i_1}\cdots z_{i_{m-1}} w_{k}\\
&=& \sum_{\bj\in \mathcal{J}(m-1,n)}
\sum_{\bi\in [\bj]}
\sum_{k=1}^n a_{(\bi,k)}(A) z_{i_1}\cdots z_{i_{m-1}} w_{k}
\\
&=& \sum_{\bj\in \mathcal{J}(m-1,n)}
\sum_{k=1}^n
\left( \sum_{\bi\in [\bj]}
 a_{(\bi,k)}(A) z_{i_1}\cdots z_{i_{m-1}}\right) w_{k}
\\
&=& \sum_{\bj\in \mathcal{J}(m-1,n)} \sum_{k=1}^n \left(a_{(\bj,k)}(A)|\bj|  z_{j_1}\cdots z_{j_{m-1}}
\right)w_{k}\,.
\end{eqnarray*}
Now note  that for every $\bj \in \mathcal{J}(m-1,n)$ we have
$|(\bj,k)| \leq m  |\bj|\,,
$
and hence by the  preceding corollary
for such  $\bj$

\begin{align*}
\left(\sum_{k: j_{m-1} \leq k} \left|c_{(\bj,k)}(P)\right|^{r'}\right)^{1/r'}
&
=
\left(\sum_{k: j_{m-1} \leq k} \left|a_{(\bj,k)}(A)|(\bj,k)|\right|^{r'}\right)^{1/r'}
\\&
\leq
m \,\left(\sum_{k=1}^n \left|a_{(\bj,k)}(A)    |\bj|\right|^{r'}\right)^{1/r'}
\leq  m \,e^{\frac{m-1}r}{|\bj|}^{1/r} \|Q\|_\infty\,.
\end{align*}
Finally, by Harris' polarization formula we know that $\|Q\|_\infty \leq e\|P\|_\infty$, and hence we  obtain the desired conclusion.
\end{proof}

\noindent
Now we are ready to give the
\begin{proof}[Proof of Theorem \ref{THMMONOMIAL} for $1 \leq r\leq 2$.]
     Take
    $P\in\mathcal P(^{\mathcal{J}(m,n)}\ell_r) $, $J \subset \mathcal{J}(m,n)$ and  $u \in \ell_r$. Then, by Lemma \ref{LEMPOLY},  for any $\bj\in J^*$,
    $$\left(\sum_{k:\ (\bj,k)\in J} |c_{(\bj,k)}(P)|^{r'}\right)^{1/r'}\leq\left(\sum_{k=j_{m-1}}^n |c_{(\bj,k)}(P)|^{r'}\right)^{1/r'}\leq m e^{1+\frac{m-1}r}{|\bj|}^{1/r} \|P\|_\infty.$$
     Now by H\"older's inequality (two times) and the multinomial formula we have
    \begin{eqnarray*}
         \sum_{\bj\in J} |c_{\bj}(P)| |u_\bj| &=&
         \sum_{\bj\in J^*} \sum_{k:\ (\bj,k)\in J} |c_{(\bj,k)}| |u_\bj| |u_k|\\&\leq&
        \sum_{\bj\in J^*}|u_\bj| \,\, \left(\sum_{k:\ (\bj,k)\in J} |c_{(\bj,k)}|^{r'}\right)^{1/r'}
        \left(\sum_{k} |u_k|^r\right)^{1/r}\\
        &\leq&m e^{1+\frac{m-1}r}\sum_{\bj\in J^*} {|\bj|}^{1/r}|u_\bj| \|u\|_r \|P\|_\infty\\
        &\leq&m e^{1+\frac{m-1}r} \left(\sum_{\bj\in J^*} {|\bj|}|u_\bj|^r\right)^{1/r}\left(\sum_{\bj\in J^*} 1\right)^{1/r'} \|u\|_r \|P\|_\infty\\
        &\leq&m e^{1+\frac{m-1}r} \left(\sum_{\bj\in \mathcal{J}(m-1,n)} {|\bj|}|u_\bj|^r\right)^{1/r}\left(\sum_{\bj\in J^*} 1\right)^{1/r'} \|u\|_r \|P\|_\infty\\
        &=&m e^{1+\frac{m-1}r} |J^*|^{1-\frac 1r}\|u\|_r^m \|P\|_\infty\,.
    \end{eqnarray*}
    In order to deduce \eqref{REMMONOMIAL1}, note that for every finite  $J \subset \mathcal{J}(m)$ there is $n$ such that
    $J \subset \mathcal{J}(m,n)$. Then every $P \in \mathcal{P}(^J \ell_r)$ can be considered as a polynomial in    $\mathcal{P}(^{\mathcal{J}(m,n)} \ell_r)$  with equal norm, which implies the conclusion.
\end{proof}

\subsection{The case  $r\geq 2$}
Note first that the simple argument from the proof of  Lemma \ref{lem:the_trick_poly} shows that we only have to deal with the case $r=\infty$.
 For  $r=\infty$ we  need another lemma  which substitutes the argument (by  Cauchy's estimates) from  Lemma \ref{coef}. It is an improvement of Parseval's identity, and its  proof can be found in
\cite[Lemma 2.5]{BaDeFrMaSe14}.

\begin{lemma} \label{mixed}
 Let $P\in \mathcal P(^{\mathcal J(m,n)} \ell_\infty)$. Then
$$\sum_{k=1}^n \left(\sum_{\substack{\bj\in\mathcal J(m-1,n)\\ j_{m-1}\leq k}}|c_{(\bj,k)}(P)|^2\right)^{1/2}\leq em2^{\frac{m-1}{2}}\|P\|_\infty.$$
\end{lemma}

\noindent
We are now ready for the

\begin{proof}[Proof of Theorem \ref{THMMONOMIAL} for $r=\infty$]
    Let $P\in \mathcal P(^{\mathcal J(m,n)} \ell_\infty)$. Then, for any
    $u\in B_{\ell_\infty}$, by the Cauchy-Schwarz inequality and the preceding Lemma \ref{mixed} we have
    \begin{eqnarray*}
        \sum_{\bj\in J}|c_{\bj}(P)||u_{\bj}|&\leq&\sum_{k=1}^n \left(\sum_{\substack{\bj\in J^*\\ (\bj,k)\in J}}|c_{(\bj,k)}|\right)\\
        &\leq&\sum_{k=1}^n \left(\sum_{\substack{\bj\in J^*\\ (\bj,k)\in J}}|c_{(\bj,k)}|^2\right)^{1/2} \left|\{\bj \in J^*: (\bj,k) \in J\}\right|^{1/2}\\
        &\leq&\sum_{k=1}^n\left(\sum_{\substack{\bj\in\mathcal J(m-1,n)\\ j_{m-1}\leq k}}|c_{(\bj,k)}|^2\right)^{1/2}
        |J^*|^{1/2}\\
        &
        \leq& em2^{\frac{m-1}{2}} |J^*|^{1/2}\|P\|_\infty.
    \end{eqnarray*}
    For the second statement, see again the argument from the proof in the case $1 \leq r \leq 2 $. This finally completes the proof of Theorem \ref{THMMONOMIAL}.
\end{proof}

\begin{remark}
 It is natural to ask for lower bounds of $\chimon\big(\pjlr\big)$ using $|J|$ or $|J^*|$. For the whole set of $m$-homogeneous polynomials, this has been done in
\cite{DGM03} for $r\geq 2$ and in \cite{BAYMAXMOD} for $1\leq r\leq 2$. Using the Kahane-Salem-Zygmund inequality, we can give such a lower bound at least for the case
$r=\infty$. Indeed, assume that $J\subset\mathcal J(m,n)$. Then there exists some absolut constant $C>0$ and signs $(\veps_\bj)_{\bj\in J}$ such that
$$\sup_{u\in B_{\ell_\infty^n}} \left|\sum_{\bj\in J}\veps_{\bj}u_\bj\right|\leq C n^{1/2}|J|^{1/2}(\log m)^{1/2}.$$
Now, the inequality
$$|J|=\sup_{u\in B_{\ell_\infty^n}} \sum_{\bj\in J}|\veps_\bj||u_\bj|\leq\chimon\big(\mathcal P(^J\ell_\infty)\big)\sup_{u\in B_{\ell_\infty^n}} \left|\sum_{\bj\in J}\veps_{\bj}u_\bj\right|$$
yields
$$\chimon\big(\mathcal P(^J\ell_\infty)\big)\geq \frac{|J|^{1/2}}{Cn^{1/2}(\log m)^{1/2}}.$$
However, the inequality given by Theorem \ref{THMMONOMIAL} is very bad if $J$ involves many independent variables;
see in particular \eqref{num3}.
\end{remark}

\begin{remark}

Given  an index set $J \subset \mathcal{J}$, we define  the Bohr radius of a  Reinhardt $R$ in $\mathbb{C}^n$ with respect to $J$ by
\begin{equation*}
    K(R;J)
        = \sup \Big\{
            0 \leq r \leq 1
        \,\big\vert\,
            \forall f \in H_\infty(R):
            \sup_{u \in rR} \,
                \sum_{\bj \in J}
                    \abs[\big]{ c_\bj(f) u_\bj }
                \leq \norm[]{ f }_\infty
        \Big\}\,.
\end{equation*}
The standard multi-variable Bohr radius is then denoted by  $K(R) =K(R;\mathcal{J}) $. Let us recall the two most important results on Bohr radii: For the open unit disc $R= \mathbb{D}$, Bohr's power series theorem states that $K(\mathbb{D}) = \frac{1}{3}$, and in \cite{BaPeSe13}(following the main idea of \cite{DFOOS}) it was recently proved that
\begin{equation*} \label{limit}
\lim_{n \rightarrow \infty} \frac{K(B_{\ell^n_\infty})}{ \sqrt{\frac{\log n}{n}}} =1\,.
\end{equation*}
For every $1 \leq r \leq \infty$ and every $n$ (with constants depending on $r$ only) we have
\begin{equation} \label{newproof}
K(B_{\ell_r^n})\,\,\asymp \,\,\Bigg( \frac{\log n}{n}\Bigg)^{1- \frac{1}{\min\{r,2\}}}\,.
\end{equation}
The probabilistic argument for the  upper estimate is  due to \cite{Boas00} (see also \cite{DGM03}), and
the proof of the lower estimate from \cite{DeFr11} uses symmetric tensor products and local Banach space theory. We here sketch  a simplified
argument based on Theorem \ref{THMMONOMIAL}.

\begin{theorem}
Let $1 \leq r \leq \infty$ and $\sigma = 1 -\frac{1}{\min(r,2)}$. Then  there is a constant $C=C(r)$ such for every
$J \subset \mathcal{J}$ and every $n$
\begin{equation}\label{caratheodory}
\displaystyle
 \frac{C}{\sup_m \abs[\big]{ (J(m,n)^*  }^{\frac{\sigma}{m}}} \leq K(B_{\ell^n_r}; J)\,,
\end{equation}
where  $J(m,n):= J \cap \mathcal{J}(m,n)$ and $C \ge \tfrac{1}{3 e^2 } > 0$.
\end{theorem}

\begin{proof}
    By a  simple analysis of   \cite[Theorem 2.2]{DGM03} as well as \cite[Lemma 2.1]{DGM03} we have
    \begin{equation*}  \label{reduction}
        \frac{1}{3}  \inf_m  K\left(B_{\ell^n_r}; J(m,n)\right) \, \leq K(B_{\ell^n_r};J) \,,
    \end{equation*}
    and
    \begin{equation*}  \label{unc-bohr}
        K \big( B_{\ell^n_r}; J(m,n) \big)
            = \frac{1}{
                \sqrt[m]{ \chimon\big( \mathcal{P}(^{J(m,n)} \ell_r) \big) }
            }\,.
    \end{equation*}
    Then the conclusion is an immediate consequence of Theorem \ref{THMMONOMIAL} and the simple fact that for the constant $C(m,r) \leq e me^{(m-1)/ \min\{ r,2\}} \le e^{2m}$.
\end{proof}

\noindent
Now the proof for the  lower bound in \eqref{newproof} follows from the the  special case $J=\mathcal{J}$. Indeed,
\begin{equation*}
    J^\ast (m,n) =\mathcal{J}(m-1,n)=\binom{(m-1)+n-1}{m-1} \leq e^{m-1} \Big(1 + \frac{n}{m-1} \Big)^{m-1}\,,
\end{equation*}
hence inserting this estimate into \eqref{caratheodory} and minimizing over $m$ gives what we want.

\end{remark}

\section{The Konyagin-Queff\'elec method} \label{KQ-method}
We now apply Theorem~\ref{THMMONOMIAL} to a special method of summation which was originally used by  Konyagin and  Queff\'elec to find the correct asymptotic order of the Sidon constant of Dirichlet polynomials of lenght $x$. In \cite{KONQUEFF} they proved the following:
there exists a constant $\beta > 0$, such that for every Dirichlet polynomial $\sum_{n=1}^x a_n n^{-s}$,
\begin{equation}
    \label{beginning}
    \sum_{n=1}^x \abs[]{ a_n }
    \leq \sqrt x \exp\left(
        \big( -\beta + o(1) \big)
        \sqrt{\log x \log\log x}
    \right)
        \sup_{t\in\mathbb R}
        \abs[\Big]{ \sum_{n=1}^x a_n n^{it} }\,.
\end{equation}
This was improved in \cite{DLB08}, where an improved  lower bound on $\beta$ was given, and in \cite{DFOOS} where the precise value $\beta = \tfrac{1}{\sqrt 2}$ was determined.

It turns out that \eqref{beginning} is linked to our subject by the Bohr point of view. Indeed, define for each $x$ the index set $J(x):=\left\{ \bj \in \mathcal{J} : p_\bj\leq x \right\}$. Then to each  Dirichlet polynomial
$$D(s)=\sum_{n=1}^x a_n n^{-s}=\sum_{\bj \in J(x)}a_{p_\bj}p_\bj^{-s}\,,$$
 we can associate a polynomial
 $$P(z)=\sum_{\bj \in J(x)}a_{p_\bj}z_\bj  \in \mathcal{P}(^J \ell_\infty) \,.$$
 Kronecker's theorem ensures that $\|P\|_\infty=\sup_{t\in\mathbb R}|D(it)|$, and the result of \cite[Theorem 3]{DFOOS}  translates into the following remarkable equality   (the improvement of \eqref{beginning}  mentioned above):
\begin{equation}\label{EQKQ1}
    \chimon \big( \mathcal{P}(^{J(x)}\ell_\infty)\big)
    = \sqrt{ x }
        \exp\left(
            \Big( - \frac {1}{\sqrt 2} + o(1) \Big)
            \sqrt{\log x\log \log x}
        \right)\,;
\end{equation}
in other terms, the latter expression gives the precise asymptotic order of the Sidon constant $S(x)$ for the characters $z_\bj\,, \bj \in J(x)$ on the group $\mathbb{T}^\infty$.

There is also an $m$-homogeneous version of \eqref{EQKQ1} due to Balasubramanian, Calado and Queffel\'{e}c  \cite{BaCaQu06} with an original  formulation analog to \eqref{beginning}. We reformulate it as follows: Define for $m$ the index set $J(x,m):=\left\{ \bj \in \mathcal{J}(m) : p_\bj\leq x \right\}$. Then with constants
only depending on $m$
\begin{equation}\label{BCQ}
\chimon \big( \mathcal{P}(^{J(x,m)}\ell_\infty)\big) \asymp \frac{x^{\frac{m-1}{2m}}}{\left(\log x\right)^{\frac{m-1}{2}}}\,.
\end{equation}

The following two theorems extend these results to the scale of $\ell_r$-spaces, and more. The original proofs
of \eqref{EQKQ1} and \eqref{BCQ} are heavily based on the Bohnen\-blust-Hille inequality and its recent improvements. Here we need Theorem~\ref{THMMONOMIAL} as a substitute. Part (1) of the first theorem
obviously extends the upper estimate from \eqref{EQKQ1} to the scale of $\ell_r$-spaces, part (2) even modifies the index set $J(x,m)$ (so far defined via  primes).
Both results are  of particular interest for our study of sets of monomial convergence in the next section.
\begin{theorem} \label{unc-basis-q}
    Let $1 \le r \le \infty$ and set  $\sigma = 1 - \tfrac{1}{\min \{ r,2 \}}$. Then for every  $f
     \in H_\infty(B_{\ell_r})$, every $u \in B_{\ell_r}$, and every  $x>e$ we have
    \begin{enumerate}
        \item[(1)] for $p$ denoting the sequence of primes,
            \begin{equation*}
                \sum_{\bj:p_\bj \le x}  \left|c_\bj(f) u_\bj \right|
                \le x^\sigma \exp \Big(
                        \big( - \sqrt{2} \sigma + o(1) \big) \sqrt{\log x\log \log x}
                    \Big)
                    \,
                    \norm{f}_\infty\,.
            \end{equation*}

        \item[(2)] for $q = (q_k)_k$, defined by $q_k = k \cdot \big( \log (k+2) \big)^\theta$ with some $\theta \in (\tfrac{1}{2},1]$,
            \begin{equation*}
                \sum_{\bj:q_\bj \le x} \left|c_\bj(f) u_\bj\right|
                \le x^\sigma \exp \Big(
                        \big( - 2 \sigma \sqrt{\theta -\frac{1}{2}} + o(1) \big) \sqrt{\log x\log \log x}
                    \Big)
                    \,
                    \norm{f}_\infty\,.
            \end{equation*}

               \end{enumerate}
    \noindent
    In both cases, the $o$-term depends  neither on $x$ nor on $f$.
\end{theorem}

\noindent
The second theorem extends \eqref{BCQ} to the scale of $\ell_r$-spaces.

\begin{theorem}\label{unc-basis-pol}
Let $1\leq r\leq \infty$ and $m\geq 1$. Then there exists $C(m,r)>0$ such that for all $P \in \mathcal{P}(^m\ell_r)$,  all $x\geq 3$,  and all $u\in \ell_r$ we have
\begin{itemize}
\item[(1)] in the case $1 \leq r \leq 2$:
$$\sum_{\bj:p_\bj\leq x}\left|c_\bj(P) u_\bj\right|\leq C(m,r)\frac{x^{\frac{m-1}{m}(1 -\frac{1}{r})}\left(\log \log x\right)^{(m-1)(1 -\frac{1}{r})} }{(\log x)^{(1 -\frac{1}{r})}}
\|u\|_r^m\|P\|_\infty\,,
$$
\item[(2)]
and in the case $2 \leq r \leq \infty$:
$$\sum_{\bj:p_\bj\leq x}
\left|c_\bj(P) u_\bj\right|\leq C(m,r)\frac{x^{\frac{m-1}{2m}}}{\left(\log x\right)^{\frac{m-1}{2}}}
  \|u\|_r^m\|P\|_\infty \,.
$$
\end{itemize}
\end{theorem}
Clearly all these results have reformulations in terms of unconditional basis constants.
For example, part (1) of Theorem \ref{unc-basis-q}  reads:
\[
\chimon \big( \mathcal{P}(^{J(x)}\ell_r)\big) \leq
x^\sigma \exp \Big(
                        \big( - \sqrt{2} \sigma + o(1) \big) \sqrt{\log x\log \log x}
                    \Big)\,.
\]
The proofs will be given in \eqref{proofproof}; the next section prepares them.

\subsection{Size of some index sets} \label{sizesizesize}
Although we stated Theorem \ref{unc-basis-q} for the sequence of primes $p$ in part (1) and for a specific choice of $q$ in part (2) we want to state our considerations below as generic as possible. Let hereinafter $q = (q_k)_k$ denote a strictly increasing sequence with $q_1 > 1$ and $q_k \to \infty$ for $k \to \infty$. For technical reasons we have to introduce the index of length zero $\vartheta = (\ )$, for which $q_\vartheta = 1$ and $(\bi, \vartheta) = (\vartheta, \bi) = \bi$ by convention. Let $x > 2$ and $2 < y < x$. Choose $l \in \N$, such that $q_l \le y < q_{l+1}$. We define
\begin{align*}
J(x)&:=\big\{\bj \in\mathcal J\,\big\vert\, q_{\bj}\le x\big\}
    \cup \{ \vartheta \} \\
    J^-(x;y) &:= \big\{ \bj = (j_1, \dotsc, j_k) \in\mathcal J(k) \,\big\vert\, k \in \N, q_\bj \le x, j_k \le l \big\}  \cup \{ \vartheta \}
\intertext{and for $m\in\N$,}
J(x,m)&:= \big\{ \bj = (j_1, \dotsc, j_m) \in \mathcal J(m) \,\big\vert\, q_\bj \le x\big\}\\
    J^+(x, m;y) &:= \big\{ \bj = (j_1, \dotsc, j_m) \in 
    \mathcal J(x,m) \,\big\vert\, l < j_1
    \big\},
\end{align*}
respectively for $m = 0$, $J^+(x,0;y) := \{ \vartheta \}$.

From the general construction of these sets we can already say something about their size -- we need five lemmas.
\begin{lemma}
    \label{lem:size_j-}
    Let $2<y<x$ and $m \in \mathbb{N}$.
    \begin{itemize}
    \item[(1)]
        $\abs{ J^-(x;y) } \le \Big( 1 + \frac{\log x}{\log q_1} \Big)^l$
     \item[(2)]
        $\abs{ J(x, m) } = \emptyset$ whenever $m >  \frac{\log x}{\log q_1}.$
    \end{itemize}
\end{lemma}

\begin{proof}
(1) Using the correspondence between $\mathcal J(m)$ and $\Lambda(m)$, $J^-(x;y)$ has the same cardinal number as
$$\Gamma^-(x;y):=\left\{\alpha\in\N_0^l \,\big\vert\,q_1^{\alpha_1}\cdots q_l^{\alpha_l}\leq x\right\}.$$
Now, for $\alpha\in\Gamma^-(x;y)$ and $1\leq j\leq l$,
$$q_1^{\alpha_j}\leq q_1^{\alpha_1}\cdots q_l^{\alpha_l}\le x,$$
so that $\alpha_j\leq\frac{\log x}{\log q_1}$ for all $j$. (2) Note that for every $\bj \in J^+(x, m;y)$ we have $q_1^m \leq q_\bj \leq x$ which immediately gives the conclusion.
\end{proof}

The next lemma relates the size of an index set with the size of its  reduced set.

\begin{lemma}
    \label{lem:size_j+_reduced}
    For the reduced index sets,
    \begin{equation*}
        J(x,m)^\ast \subset J\big( x^\frac{m-1}{m}, m-1 \big)
        \quad\text{ and }\quad
        J^+(x,m;y)^\ast \subset J^+\big( x^\frac{m-1}{m}, m-1;y \big).
    \end{equation*}
\end{lemma}

\begin{proof}
    Let $\bj = (j_1, \dotsc, j_{m-1}) \in J(x,m)^{*}$, respectively 
    $\bj \in J^+(x,m;y)^{*}$. Then there exists $k \ge j_{m-1}$ such that  $(\bj, k) \in J(x,m)$, respectively $(\bj, k) \in J^+(x,m;y)$. Hence $q_{\bj} \cdot q_{k} = q_{(\bj, k)} \le x$. Since $q_k \ge q_{j_{m-1}}$, this implies either $q_k > x^{\frac{1}{m}}$ or $q_{j_1} \le \dotsc \le q_{j_{m-1}} \le q_k \le x^\frac{1}{m}$. In both cases, $q_{j_1} \dotsb q_{j_{m-1}} \le x^\frac{m-1}{m}$.
\end{proof}

\noindent
For specific choices of $q$ we can say the following about the size of $J^+(x,m;y)$:
\begin{lemma}
    \label{thm:J+_theta}
    Let $q = (q_k)_k$ be defined by $q_k = k \cdot \big( \log (k+2) \big)^\theta$ for some $\theta \in (0,1]$. Then there exists a constant $c > 0$, such that for every $x > y > 2$ and every $m\in\N$,
    \begin{equation*}
        \abs{ J^+(x,m;y) } \le x y^{-m} \exp\Big( y \cdot \big( g_\theta(x) + c \big) \Big)
    \end{equation*}
    where $g_\theta(x) = \frac1{1-\theta}(\log x)^{1-\theta}$ for $\theta < 1$ and $g_\theta(x) = \log \log x$ for $\theta = 1$.
\end{lemma}

\begin{proof}
    From the definition of the series $q$, we see immediately
    \begin{equation}
        \label{eq:q_diff}
        q_{l+k} - q_l \ge q_k
    \end{equation}
    for any $k \in\N$. We have furthermore for $c = q_1^{-1} + q_2^{-1} + q_3^{-1}$,
    \begin{equation*}
        \sum_{k \le x} \frac{1}{q_k}
        \le \sum_{3 < k \le x} \frac{1}{k (\log k )^\theta} + c
        \le \int_3^x \frac{1}{t (\log t)^\theta} \mathrm{d}t + c
        = \int_{\log 3}^{\log x} \frac{1}{s^\theta} \mathrm{d}s +c
    \end{equation*}
    and therefore by integration
    \begin{equation}
        \label{eq:sum_q}
        \sum_{k \le x} \frac{1}{q_k} \le g_\theta(x) + c.
    \end{equation}

    \medskip
    \noindent
    We introduce a completely multiplicative function,
    \begin{eqnarray*}
        \abs{J^+(x,m;y)}
        &=& \sum_{\bj \in J^+(x,m;y)} 1
        \le \frac{x}{y^m} \sum_{\bj \in J^+(x,m;y)}
            \frac{y}{q_{j_1}} \dotsb \frac{y}{q_{j_m}}\\
        & \le& \frac{x}{y^m} \prod_{l < k < x} \bigg( \sum_{\nu = 1}^\infty \big( \frac{y}{q_k} \big)^\nu \bigg)
        \le \frac{x}{y^m}
            \exp \Big( - \smashoperator{\sum_{l < k < x}} \log \big( 1 - \frac{y}{q_k}\big) \Big).
    \end{eqnarray*}

    \noindent
    Using the series expansion of the logarithm around $1$, we obtain for the exponent
    \begin{align*}
        - \smashoperator{\sum_{l < k < x}} \log \big( 1 - \frac{y}{q_k}\big)
        &= \sum_{l < k < x} \sum_{\nu = 1}^\infty \frac{1}{\nu} \Big( \frac{y}{q_k}\Big)^\nu
        \le \sum_{l < k < x} \frac{y}{q_k} \frac{1}{1 - \frac{y}{q_k}}.
        \end{align*}
   With \eqref{eq:q_diff} and the fact that $y \ge q_l$, this leads to
   \begin{align*}
        - \smashoperator{\sum_{l < k < x}} \log \big( 1 - \frac{y}{q_k}\big) &
        \le y \sum_{l < k < x} \frac{1}{q_k - y}
        \le y \sum_{l < k < x} \frac{1}{q_{k - l}}
        \le y \cdot \big( \sum_{k < x} \frac{1}{q_k} \big).
    \end{align*}
    \eqref{eq:sum_q} now completes the proof.
\end{proof}

\noindent
Finally, we mention two known estimates which measure the size of $J(x,m)$ and $ J^+(x,m;y)$, in the case they are defined with respect to the sequence of primes. The first  one is taken from Balazard
\cite[ Corollaire~1]{Balazard89}, and the second  one is a well-known result of Landau (see e.g. \cite{HW} for a proof).

\begin{lemma}
    \label{thm:J+_primes}
    Let $q$ denote the sequence of primes. Then there exists a constant $c > 0$, such that for every $x > y > 2$ and every $m\in\N$,
    \begin{equation*}
        \abs{ J^+(x,m;y) } \le x y^{-m} \exp\Big( y \cdot \big( \log \log x + c \big) \Big).
    \end{equation*}
\end{lemma}
\vspace{3mm}

\begin{lemma}\label{sizejsets}
Let $q$ denote the sequence of primes and let $m\geq 1$. There exists a constant $C_m>0$ such that, for all $x\geq 3$,
\begin{equation}
 |J(x,m)|\leq C_m \frac{x}{\log x}(\log \log x)^{m-1} \label{EQ2}
\end{equation}
\end{lemma}

\subsection{Proofs} \label{proofproof}
The proof of Theorem \ref{unc-basis-pol} is now very short.

\begin{proof}[Proof of Theorem~\ref{unc-basis-pol}.]
   The proof of the first statement is a direct consequence of Theorem \ref{THMMONOMIAL} for the index set $J=J(x,m)$, and of  the Lemmas \ref{lem:size_j+_reduced} and  \ref{sizejsets}. The second statement follows from  Lemma \ref{lem:the_trick_poly} combined with \eqref{BCQ}.
\end{proof}

\noindent
To present the Konyagin-Queff\'elec technique in general we need one more additional lemma.
\begin{lemma}
    \label{lem:cif}
    Let $m_1, m_2, l \in\N$ and  $P\in \mathcal{P}(^{m_1 + m_2} \ell_r)$ such that  $c_{\mathbf k} (P)\neq 0$ for only finitely many
    $\mathbf k \in \mathcal{J}(m_1+m_2)$. Then for every $\bi \in \mathcal J(m_1, l)$ the polynomial
       \begin{equation*}
        P_\bi = \sum_{\substack{\bj \in \mathcal J(m_2)\\j_1 > l}} c_{(\bi,\bj)}(P)\, z_{(\bi,\bj)} \in \mathcal{P}(^{m_2} \ell_r)
    \end{equation*}
    satisfies $$\norm[]{ P_\bi }_\infty \le \norm[]{ P }_\infty\,.$$
    \end{lemma}

\begin{proof}
        Given $u \in \ell_r$,  a straightforward calculation shows
    \begin{equation*}
        P_\bi(u)
        = \int_{\mathbb{T}^l}
            P\big( \zeta_1 u_1, \ldots, \zeta_{l}u_{l},u_{l+1}\ldots\big )
            \, \bar{\zeta}_{i_1} \dotsm  \bar{\zeta}_{i_l}
            \mathrm{d} (\zeta_1 , \ldots, \zeta_{l})\,,
    \end{equation*}
    which immediately implies the desired inequality.
\end{proof}

\begin{proof}[Proof of Theorem~\ref{unc-basis-q}.]

\bigskip\noindent
Recall the setting of our theorem. Let $x > e$ and $2 < y < x$, and choose $l \in \N$ such that $q_l \le y < q_{l+1}$.
Given $u \in B_{\ell_r}$, at first write $u = u^- + u^+$ where $u^-_k = 0$ for $k > l$ and $u^+_k = 0$ for $k \le l$. Any $\mathbf k\in J(x)$ may be written as $\mathbf k=(\bi,\bj)$ with   $\bi\in J^-(x;y)$ and  $\bj\in J^+(x,m;y)$. Moreover, $|u_\bi|=|u_{\bi}^-|$ and $|u_\bj|=|u_{\bj}^+|$. Hence,
\begin{align*}
    \sum_{q_{\mathbf k} \le x} \abs{c_{\mathbf k} u_{\mathbf k}}
    &
    = \sum_{\bi \in J^-(x;y)}
    \,\,
        \sum_{m\in \mathbb{N}_0}
         \,\,
        \sum_{\substack{\bj\in J^+(x,m;y) \\ q_{(\bi,\bj)} \le x}}
            \abs{c_{(\bi,\bj)} u_{(\bi,\bj)}}
            \\&
    = \sum_{\bi \in J^-(x;y)}
    \,\,
        \sum_{m\in \mathbb{N}_0}
        \abs{u^-_\bi}
          \,\,
        \sum_{\substack{\bj\in J^+(x,m;y) \\ q_{(\bi,\bj)} \le x}}    \abs{c_{(\bi,\bj)} u^+_\bj}\,.
\end{align*}
Using Theorem~\ref{THMMONOMIAL}, we can now estimate the latter sum for every $\bi \in J^-(x;y)$,
    \begin{align*}
        \abs{u^-_\bi}
          \sum_{\substack{\bj\in J^+(x,m;y) \\ q_{(\bi,\bj)} \le x}}
            \abs{c_{(\bi,\bj)} u^+_\bj}
        &\le \abs{u^-_\bi}
             C^m
            \abs{J^+(x,m;y)^\ast}^\sigma
            \sup_{\substack{\norm{\zeta}_{r} \le \norm{u^+}_r \\  \forall k \le l: \zeta_k = 0 }}
            \abs[\Big]{
                \sum_{\substack{\bj\in \mathcal J(m) \\ j_1 > l}}
                c_{(\bi,\bj)} \zeta_\bj
            } \\
        &\le C^m \abs{J^+(x,m;y)^\ast}^\sigma
            \sup_{\substack{\norm{\zeta}_{r} \le \norm{u^+}_r \\  \forall k \le l: \zeta_k = 0 }}
            \abs[\Big]{
                \sum_{\substack{\bj\in \mathcal J(m) \\ j_1 > l}}
                c_{(\bi,\bj)} u^-_\bi \zeta_\bj
            } \\
        &\le C^m \abs{J^+(x,m;y)^\ast}^\sigma
            \norm[\Big]{
                \sum_{\substack{\bj\in \mathcal J(m) \\ j_1 > l}}
                c_{(\bi,\bj)} z_{(\bi,\bj)}
            }_\infty,
    \end{align*}
    where the last inequality is a consequence of $( u^- + \zeta )_{(\bi,\bj)} = u^-_\bi \zeta_\bj$ and
    $$\norm{ u^- + \zeta }_r^r = \norm{u^-}^r_r + \norm{\zeta}_r^r \le \norm{ u^- }^r_r + \norm{u^+}^r_r \le 1\,.$$
    Choose for each $\bi \in J^-(x,y)$ some $m_\bi \in \mathbb{N}$ such that $\bi \in \mathcal{J}(m_\bi)$.
    By Lemma~\ref{lem:cif} we then  obtain
    \begin{align*}
        \sum_{q_{\mathbf k} \le x} \abs{c_{\mathbf k} u_{\mathbf k}}
        &\le \sum_{\bi \in J^-(x;y)} \quad \sum_m C^m \abs{J^+(x,m;y)^\ast}^\sigma
        \quad
            \norm[\big]{
                \;\;
                                \smashoperator{ \sum_{ \mathbf k \in \mathcal J(m+m_{\bi}) } }
                \;\;
            c_{\mathbf k} z_{\mathbf k}}_\infty\,.
            \end{align*}
           Moreover, if we decompose $f$ into its sum of homogeneous Taylor polynomials, then we deduce  by Cauchy estimates that
                      \begin{align*}
              \sum_{q_{\mathbf k} \le x} \abs{c_{\mathbf k} u_{\mathbf k}}
        &\le \Big(
            \abs{J^-(x;y)} \sum_m C^m \abs{J^+(x,m;y)^\ast}^\sigma
            \Big)
            \norm{ f }_\infty.
    \end{align*}
    Now $J^+(x,m;y)^\ast \subset J^+(x^\frac{m-1}{m}, m-1)$ and $J^+(x,m;y) = \emptyset$ for $m > \tfrac{ \log x }{ \log q_1 }$ by Lemma \ref{lem:size_j+_reduced} and Lemma~\ref{lem:size_j-}.
    Hence
    \begin{equation*}
        \abs{J^-(x;y)} \cdot \sum_m C^m \abs{J^+(x,m;y)^\ast}^\sigma \\
        \le \Big(1 + \frac{ \log x }{ \log q_1 } \Big)^{l+1}
            \sup_m C^m \abs{J^+(x^\frac{m-1}{m},m-1)}^\sigma.
    \end{equation*}
    Up to this point, our arguments are independent of the  specific choice of $q$. We threat both cases at once. In the case of $q$ denoting the sequence of primes, set $\theta = 1$.
    By Lemma~\ref{thm:J+_theta} and Lemma~\ref{thm:J+_primes}, respectively
    \begin{align*}
        \MoveEqLeft[3] \Big( 1 +\frac{ \log x }{ \log q_1 } \Big)^{l+1}
            \cdot
            \sup_m C^m \abs{J^+(x^\frac{m-1}{m},m-1)}^\sigma \\
        &\le \Big( 1 +\frac{ \log x }{ \log q_1 } \Big)^{l+1}
            \cdot
            \sup_m \bigg( C^m x^{\frac{m-1}{m}} y^{-m+1}
                \exp \Big( y \cdot \big( g_\theta(x) + c \big) \Big) \bigg)^\sigma. \\
    \intertext{Choosing $y = \tfrac{(\log x)^{\theta - \frac{1}{2}}}{\log \log x}$, this is}
        &= x^\sigma \exp\Big( o(1) \sqrt{\log x \log \log x} \Big)
            \cdot \sup_m \big( \overbrace{ C^m x^{-\frac{1}{m}} y^{-m} }^{=:\; \exp h_{x,y}(m)} \big)^\sigma.
    \end{align*}
    Note that $l = O(1) \tfrac{y}{(\log y)^\theta} = o(1) \tfrac{\sqrt{\log x}}{\log \log x}$\,; \,indeed,
     by the definition of $l$
\begin{equation*}
            l
            \, \big( \log \big( l + 2 \big) \big)^\theta
        \le y
        <   \big( l + 1 \big)
            \, \big( \log \big( l + 3 \big) \big)^\theta
        \le \big( l + 2 \big)^2
        \, ,
    \end{equation*}
   hence
    \begin{equation*}
        \frac{ y }{ ( \log y )^\theta }
        \ge \frac
        {
            l
            \, \big( \log \big( l + 2 \big) \big)^\theta
        }
        { \big( \log \big( ( l + 2 )^2 \big) \big)^\theta }
        =   2^{-\theta} \, l
        \, .
    \end{equation*}
    Differentiating $$h_{x,y}(m) = m\log C - \tfrac{1}{m}\log x - m \log y\,,$$ we see that it attains its maximum at
    \begin{equation*}
        M = \sqrt{\frac{\log x}{\log y - C}}
          \ge \sqrt{\frac{\log x}{\log y}},
    \end{equation*}
    and therefore
    \begin{align*}
        h_{x,y}(m)
        &\le h_{x,y}(M) \\
        &= \underbrace{\log(C) \sqrt{\frac{\log x}{\log y - C}} }_{
            \mathclap{=\; o(1)\sqrt{\log x \log \log x}}
          }
          - 2 \sqrt{\log x \log y} \\
        &= \big( -2 \sqrt{\theta - \tfrac{1}{2}} + o(1) \big) \sqrt{ \log x \log \log x},
    \end{align*}
    which proves the theorem.
\end{proof}

\section{Monomial convergence}
In this section we apply the  new estimates on the unconditional basis constant of polynomials on $\ell_r$
from the preceding two sections, to the  analysis of  sets $\mon \mathcal{P}(^m\ell_r)$
and $\mon H_\infty(B_{\ell_r})$ of monomial convergence of $m$-homogeneous polynomials on $\ell_r$ and bounded holomorphic functions on $B_{\ell_r}$.

\subsection{Polynomials } \label{polynomials}

The next statement gives the state of art for homogeneous polynomials.
\begin{theorem}Let $1 \leq r \leq \infty$ and $m\geq 2$.
	\label{zero}
	\hfill
	\begin{enumerate}
		\item[(1)]
			If $r=\infty$, then
			$\mon \mathcal{P}(^m \ell_\infty) = \ell_{\frac{2m}{m-1}, \infty}$.
		\item[(2)]
			If $r=1$, then
			$\mon \mathcal{P}(^m \ell_1) = \ell_{1}$.
		\item[(3)]
			If $\,2 \le r  <\infty$, then
			$\ell_{\frac{2m}{m-1}, \infty} \cdot \ell_r \subset \mon \mathcal{P}(^m \ell_r) \subset \ell_{\big(\frac{m-1}{2m} + \frac{1}{r}\big)^{-1}, \infty}$.
		\item[(4)]
			If $1 < r < 2$, then for any $\veps>0$,
			$\ell_{(mr')'- \varepsilon} \subset \mon \mathcal{P}(^m \ell_r) \subset \ell_{(mr')', \infty}$.
	\end{enumerate}
\end{theorem}

\noindent
Several cases of this theorem are already known: the first one can be found in \cite{BaDeFrMaSe14} and the second one in \cite{DeMaPr09}. The upper estimate in the third and the fourth case can also be found in \cite{DeMaPr09}.
 The proof of the lower estimate in the third case follows from a general technique inspired by Lemma \ref{lem:the_trick_poly}. We need to introduce another  notation. For $X$ a Banach sequence space,  $R$ a Reinhard domain in $X$ and $\mathcal F(R)$ a set of holomorphic functions on $R$, we set
 $$[\mathcal F(R)]_\infty=\big\{f_w:u\in B_{\ell_\infty}\mapsto f(uw);\ w\in R,\ f\in\mathcal F(R)\big\}.$$
 $[\mathcal F(R)]_\infty$ is a set of holomorphic functions on $B_{\ell_\infty}$, and the following general result holds true.

\begin{lemma}\label{LEMESTIMGENERALE}
$R\cdot \mon[\mathcal F(R)]_\infty\subset\mon\mathcal F(R)$.
\end{lemma}

\begin{proof}
Let $w\in R$ and $u \in \mon [\mathcal F(R) ]_\infty$. For any $f\in\mathcal F(R)$ then $c_\alpha(f_w)=w^\alpha c_\alpha(f)$ and therefore
$$ \sum_{\alpha}|c_\alpha(f)||wu|^\alpha= \sum_\alpha |c_\alpha(f_w)||u|^\alpha<+\infty
.
$$
which yields the claim.
\end{proof}

\noindent
It is now easy to deduce the lower estimate in the third case, knowing the result of part (1). Indeed, $[\mathcal P(^m X)]_\infty$ is contained in the set of bounded $m$-homo\-geneous polynomials on $B_{\ell_\infty}$, thus in $\mathcal P(^m\ell_\infty)$ by the natural extension of a bounded polynomial from $B_{\ell_\infty}$ to $\ell_\infty$.

\medskip

The lower inclusion in (4) is a partial solution of a conjecture made in \cite{DeMaPr09} (see the remarks after Example 4.6 in \cite{DeMaPr09}). Its proof  seems less simple, and requires some preparation. Note that for $r \ge 2$ we
have that \begin{equation*}
		\frac{1}{
			p^{\frac{m-1}{2m}}
		} \cdot \ell_r
		\subset \mon \mathcal{P}(^m \ell_r)
	\end{equation*}
 which is an immediate consequence of Theorem \ref{zero}, (3). For  $1 <r< 2$ we can prove this  up to an $\varepsilon$:

\begin{theorem}\label{LEMPRESQUEPOLY}
 For $1 <r< 2$ and  $m\geq 1$ put $\sigma_m=\frac{m-1}{m}\left(1-\frac 1r\right)$. Then for every $\veps>\frac 1r$
$$\frac{1}{p^{\sigma_m}\big(\log(p)\big)^{\veps}}\cdot \ell_r\subset \mon\mathcal P(^m\ell_r).$$
In particular, for all $\varepsilon>0$,
	\begin{equation*}
		\frac{1}{
			p^{\sigma_m + \varepsilon }
		} \cdot \ell_r
		\subset \mon \mathcal{P}(^m \ell_r).
	\end{equation*}
\end{theorem}

\begin{proof}

Let $P=\sum_{\bj\in\mathcal J(m)}c_\bj(P) z_\bj\in\mathcal P(^m\ell_r)$ and let $u\in\ell_r$. We intend to show that
\begin{equation*}
    S:=\sum_{\bj\in\mathcal J(m)} |c_{\bj}(P)| \frac{1}{(p_{j_1}\cdots p_{j_m})^{\sigma_m} \big( \log(p_{j_1})\cdots \log(p_{j_m})\big)^\veps}|u_\bj|
    \leq C \norm[]{ u }_r^m \norm[]{ P }_\infty
\end{equation*}
for some constant $C>0$.
 Let us observe that, for any $j_1,\dots,j_m\geq 1$,
\begin{equation}\label{EQ1}
 \log(p_{j_1})\cdots \log(p_{j_m})\geq \frac{(\log 2)^{m-1}}m\ \log(p_{j_1}\cdots p_{j_m}).
\end{equation}
 We order the sum over $\bj\in\mathcal J(m)$ with respect to the value of the product $p_{j_1}\cdots p_{j_m}$.
Precisely, using (\ref{EQ1}), we write
\begin{eqnarray*}
 S&\ll&\sum_{N=m}^{+\infty} \sum_{\substack{\bj\in\mathcal J(m)\\ 2^N\leq p_\bj<2^{N+1}}}\frac{1}{p_\bj^{\sigma_m} \log^\veps(p_\bj)}|c_\bj(P)| |u_\bj|
\ll\sum_{N=m}^{+\infty}\frac{1}{2^{N\sigma_m}N^{\veps}}\sum_{p_\bj\le 2^{N+1}} |c_\bj(P)| |u_\bj|.
\end{eqnarray*}
We apply Theorem \ref{unc-basis-pol} to find
$$S\ll\sum_{N=m}^{+\infty}\frac{1}{2^{N\sigma_m}N^{\veps}} \frac{2^{N\sigma_m} \log(N)^{(m-1)\left(1-\frac 1r\right)}}{N^{1-\frac 1r}}\|P\|_\infty\|u\|_r^m.$$
The series is convergent since $\veps>1/r$.
\end{proof}

\noindent
Finally, we are ready to provide  the
\begin{proof}[Proof of the lower inclusion of Theorem~\ref{zero}, (4)]
	Given  $u \in  \ell_{(mr')'- \varepsilon}$, we show  that the decreasing rearrangement $u^* \in  \mon \mathcal{P}(^m \ell_r)$.
	Then for some $\delta > 0$ we have
	\begin{equation*}
		u^*_n \ll  \frac{1}{n^{\frac{1}{(mr')'- \varepsilon}  } } =  \frac{1}{n^{\frac{1}{(mr')'} + \delta } }.
	\end{equation*}
	By the prime number theorem we know that
	$p_n\asymp n \log n$,
		hence
	\begin{equation*}
		\frac{1}{
			n^{ \frac{1}{(mr')'} + \delta }
		}
		=
		\frac{1}{
			p_n^{
				\frac{m-1}{m} \frac{1}{r'}
				+ \frac{\delta}{2}
			}
		}
		\frac{
			p_n^{
				\frac{m-1}{m} \frac{1}{r'}
				+ \frac{\delta}{2}
			}
		}{
			n^{ \frac{1}{(mr')'} + \delta }
		}
		\ll
		\frac{1}{
			p_n^{
				\frac{m-1}{m} \frac{1}{r'}
				+ \frac{\delta}{2}
			}
		}
		\frac{
			(n \log n)^{
				\frac{m-1}{m} \frac{1}{r'}
				+ \frac{\delta}{2}
			}
		}{
			n^{ \frac{1}{(mr')'} + \delta }
		}.
	\end{equation*}
	But obviously
	\begin{equation*}
		\frac{
			(n \log n)^{
				\frac{m-1}{m} \frac{1}{r'}
				+ \frac{\delta}{2}
			}
		}{
			n^{ \frac{1}{(mr')'} + \delta }
		}
		=
		\frac{
			(\log n)^{
				\frac{m-1}{m} \frac{1}{r'}
				+ \frac{\delta}{2}
			}
		}{
			n^{ \frac{1}{r} }
			n^{ \frac{\delta}{2} }
		}
		\in \ell_r,
	\end{equation*}
	hence by Theorem \ref{LEMPRESQUEPOLY}
	\begin{equation*}
		\frac{1}{
			n^{ \frac{1}{(mr')'} + \delta }
		} \in  \mon \mathcal{P}(^m \ell_r),
	\end{equation*}
	the conclusion.
\end{proof}

\begin{remark}
 A look at \cite{DeMaPr09} shows that in the case   $r>2$ the proof of the inclusion $\mon\mathcal P(^m\ell_r)\subset \ell_{\big(\frac{m-1}{2m} + \frac{1}{r}\big)^{-1}, \infty}$
keeps working if we replace $\ell_r$ by $\ell_{r,\infty}$. Indeed, it just uses that
$$\sup_{u\in \ell_r^n,\ \|u\|_r\leq 1}\sum_{k=1}^n |u_k|^2=n^{1-\frac 2r}$$
and this remains true, up to a constant factor, if we replace $B_{\ell_r^n}$ by $B_{\ell_{r,\infty}^n}$. If we combine this with Lemma \ref{LEMESTIMGENERALE},
then we find that, for $r>2$,
$$\mon \mathcal P(^m\ell_{r,\infty})=\ell_{\big(\frac{m-1}{2m} + \frac{1}{r}\big)^{-1}, \infty}.$$
\end{remark}

\subsection{Holomorphic functions} \label{holomorphic functions}
We now study $\mon H_\infty(B_{\ell_r})$ for $1\leq r\leq +\infty$. The extreme cases are already well-known: By a result of Lempert (see e.g. \cite{Le99} and \cite{DeMaPr09}) we have
\begin{equation} \label{Lempert}
\mon H_\infty(B_{\ell_1}) = B_{\ell_1}\,.
\end{equation}
Moreover by \cite{BaDeFrMaSe14} we know that
\begin{equation} \label{main}
B \subset \mon H_\infty(B_{\ell_\infty}) \subset \overline{B}
\end{equation}
where
\begin{eqnarray*}
B& =& \Bigg\{
			u \in B_{\ell_\infty}
		\,;\,
			\limsup_n \frac{1}{\log n} \sum_{k=1}^n \lvert u_k^\ast \rvert^2 < 1
		 \Bigg\}\,\\
\overline B &=&  \Bigg\{
			u \in B_{\ell_\infty}
		\,;\,
			\limsup_n \frac{1}{\log n} \sum_{k=1}^n \lvert u_k^\ast \rvert^2 \leq 1
		 \Bigg\}\,.
\end{eqnarray*}
For $1<r<\infty$, it was shown in \cite{DeMaPr09} that, setting $\tfrac{1}{s} = \tfrac{1}{2} + \tfrac{1}{\max\{r,2\}}$, for every $\varepsilon >0$
\begin{equation}
		\label{eq:mon_epsilon}
		B_{\ell_r} \cap \ell_s
		\subset \mon H_\infty(B_{\ell_r})
		\subset B_{\ell_r} \cap \ell_{s+\epsilon}\,.
	\end{equation}
In the following we improve the previous inclusion, and show in particular that here $\varepsilon = 0$ is not possible.
More precisely, we give necessary and sufficient conditions on $(\alpha,\beta)\in [0, \infty[^2$ such that
$$\left(\frac{1}{n^{\alpha}\big(\log (n+2)\big)^\beta}\right)_n \in \mon H_\infty(B_{\ell_r})\,.$$
Note that by \eqref{main} for every $\beta >0$
\begin{equation} \label{one}
\left(\frac{1}{n^{\frac{1}{2}}\big(\log (n+2)\big)^\beta} \right)_n\in \mon H_\infty(B_{\ell_\infty})\,;
\end{equation}
we do not know whether here $\beta=0$ is possible. Moreover,  by \eqref{Lempert}
\begin{equation} \left(\label{two}\frac{1}{n\big(\log (n+2)\big)^\beta}\right)_n \in \mon H_\infty(B_{\ell_1})
\end{equation}
if and only if  $\beta >1$. The following result collects our knowledge in the remaining cases:

\begin{theorem}
	\label{thm:psigma}
	For $1\le  r \le \infty$ put $\sigma=1-\frac{1}{\min(r,2)}$. Then
    \begin{enumerate}
        \item[(1a)]
            For any $\theta > \frac{1}{2}$ and $1\le  r \le 2$
            \begin{equation*}
                \left(\frac{1}{ n^{\sigma} \cdot \big( \log (n+2) \big)^{\theta\sigma}}\right)_n
                \cdot B_{\ell_r}
                \subset \mon H_\infty( B_{\ell_r} ).
            \end{equation*}

            In particular, $\left(\tfrac{1}{n^{\frac 1r+\sigma}(\log (n+2))^{\beta}}\right)_n \in \mon H_\infty(B_{\ell_r})$  whenever  $\beta>\frac{1}{2r}+\frac{1}{2} $.

        \item[(1b)]
            For any $\theta > 0$ and $2\le r \le \infty$
            \begin{equation*}
               \left(\frac{1}{n^\sigma \cdot \big( \log (n+2) \big)^\theta }\right)_n
                \cdot B_{\ell_r}
                \subset \mon H_\infty( B_{\ell_r} ).
            \end{equation*}

            In particular,  $\left(\tfrac{1}{n^{\frac 1r+\sigma}(\log (n+2))^{\beta}}\right)_n \in \mon H_\infty(B_{\ell_r})$ whenever $\beta>\frac 1r$.

        \item[(2)]
            Suppose that  $\displaystyle \left(\frac{1}{n^{\frac 1r+\sigma}(\log (n+2))^{\beta}}\right)_n\in \mon H_\infty(B_{\ell_r})$. Then $\beta\geq \frac{1}{r}$.
    \end{enumerate}
\end{theorem}

\noindent
Note that we cannot replace $\log(n+2)$ by $\log(n+1)$ in the previous statement. Indeed, it can be easily seen by restricting the study to the one-dimensional case that $\mon H_\infty(B_{\ell_r}) \subset \mon H_\infty(B_{\ell_\infty})\subset B_{\ell_\infty}$.

\begin{proof}
We begin with part (1b): We have that $\theta > 0$ and $2\le r \le \infty$.
    Then an easy argument yields
    \begin{equation*}
       \left( \frac{1}{n^\frac{1}{2} \cdot \big( \log (n+2) \big)^\theta }\right)_n
        \in B\,.
    \end{equation*}
   By Lemma \ref{LEMESTIMGENERALE} and by observing that $[H_\infty(B_{\ell_r})]_\infty\subset H_\infty(B_{\ell_\infty})$, we get immediately
    \begin{equation*}
       \left( \frac{1}{n^\frac{1}{2} \cdot \big( \log (n+2) \big)^\theta }\right)_n
            \cdot B_{\ell_r}
            \subset
            \mon H_\infty( B_{\ell_\infty} )
            \subset \mon H_\infty( B_{\ell_r} ).
                \end{equation*}
Let us now prove part (1a):
Assume that $\theta > \frac{1}{2}$ and $1\le  r \le 2$.
We shall apply Theorem  \ref{unc-basis-q}, (2) with the sequence $q$ defined by $q_j=j\cdot (\log(j+2)\big)^\theta$.
    Then for  $f \in H_\infty(\ell_r)$ and $u \in B_{\ell_r}$ the conclusion follows from
    \begin{align*}
        \MoveEqLeft[3]
        \sum_{\bj} \frac{1}{q_\bj^\sigma} \abs{ c_\bj(f) u_\bj }
        \\
        &= \sum_{N = 0}^\infty \;
             \sum_{ \bj: e^N < q_\bj \le e^{N+1}}
                \frac{1}{q_\bj^\sigma} \abs{ c_\bj(f) u_\bj } \\
        &\le \sum_{N = 0}^\infty
             \frac{1}{e^{N\sigma}}\;
             \sum_{ \bj\in J(e^{N+1})}
                 \abs{ c_\bj(f) u_\bj } \\
        &\le \sum_{N = 0}^\infty
                \frac{1}{e^{(N-1)\sigma}}
                e^{N\sigma}
                \exp \Big(
                    \big( -2 \sigma \sqrt{\theta - \frac{1}{2}} + o(1) \big)
                    \sqrt{N \log N} \Big)
            \cdot
            \norm{f}_\infty
        < \infty.
    \end{align*}
    Finally, we check part (2): For $r=1$, \eqref{Lempert} already proves the claim. At first we will treat the case $1< r \le 2$ with a probabilistic argument. Afterwards we reduce the case $r \ge 2$ to the case $r=2$. Let $1 < r\le 2$. We shall apply Corollary~3.2 of \cite{BAYMAXMOD} (with $p = r$). Then there is an absolute constant $C \ge 1$ such that for any $m,n$ there are $(\veps_\bj)_\bj \in \T^{\jmn}$ for which
    \begin{equation}
        \label{eq:baymaxmod}
        \sup_{ u \in B_{\ell_r^n} }
            \abs[\Big]{
                \sum_{\bj}
                    \veps_{\bj}
                    \, \abs[]{ \bj }
                    \, u_\bj
            }
        \le C (n \log m)^\sigma m^{m \sigma}
        \, .
    \end{equation}
        Let now $x = \big( k^{-1}(\log (k+2))^{-\beta} \big)_k$ denote the sequence in question and assume $x \in \mon H_\infty(B_{\ell_r})$. Then, by a closed graph argument, there exists a constant $\tilde C \ge 1$, such that for every $f \in H_\infty(B_{\ell_r})$,
    \begin{equation}
        \label{eq:cg_constant}
        \sum_{\bj}
            \abs[]{ c_\bj(f) \, x^\alpha }
        \le \tilde C \, \norm[]{ f }
        \, .
    \end{equation}
    For any $n\in\N$ now,
    \begin{equation*}
        \Big( \sum_{k=1}^n  \abs{x_k} \Big)^m
        =
            \sum_{\bj \in \jmn}
            \abs[\big]{
                \veps_\bj
                \, \abs[]{ \bj }
                \, \abs[]{ x_\bj }
            }
        \le \tilde C
            \sup_{ u \in B_{\ell_r^n} }
            \abs[\Big]{
                \sum_{\bj}
                    \veps_{\bj}
                    \, \abs[]{ \bj }
                    \, u_\bj
            }
        \le \tilde C \, C
            \, (n \log m)^\sigma m^{m \sigma}
        \, .
    \end{equation*}
    by \eqref{eq:cg_constant} and \eqref{eq:baymaxmod}. Taking the $m$\textsuperscript{th} root, we obtain
    \begin{equation*}
        \sum_{k=1}^n \frac{1}{k ( \log (k+2) )^\beta}
        \le \big( \tilde C \, C \big)^\frac{1}{m}
            (n \log m)^\frac{\sigma}{m} m^{\sigma}
    \end{equation*}
    for every $n,m\in\N$. It now suffices to notice that with $m = \lfloor \log n \rfloor$ the right-hand side is asymptotically equivalent to $(\log n)^\sigma$ and the left-hand side to $(\log n)^{1-\beta}$  as $n \to \infty$. Hence $\beta > -\sigma +1 = \tfrac{1}{r}$.

    Now suppose $r\ge 2$ and set $\xi = \big( k^{-\frac{1}{t}} (\log(k+2))^{-\frac{1}{t} - \veps} \big)_k$ for $\tfrac{1}{t}+\tfrac{1}{r}=\tfrac{1}{2}$ and $\veps > 0$. Consider $f\in H_\infty(B_{\ell_2})$ and let us set $g = f \circ D_\xi$, where $D_\xi$ denotes the diagonal operator $\ell_r \to \ell_2$ induced by $\xi$, which is bounded by Hölder's inequality. Thus $g \in H_\infty(B_{\ell_r})$. We have
    \begin{equation*}
        \sum_\bj \abs[]{ c_\bj(f) }
        \frac{1}{j_1 (\log(j_1 + 2))^{\frac{1}{t} + \beta + \veps }}
        \dotsm
        \frac{1}{j_m (\log(j_m + 2))^{\frac{1}{t} + \beta + \veps }}
        = \sum_\bj
            \abs[\big]{ \big( c_\bj(f) \, \xi_\bj \big) \, x_\bj }
        < \infty
        \, ,
    \end{equation*}
    under the assumption that $x = \big( k^{-\frac{1}{r} - \frac{1}{2}} (\log(k+2))^{-\beta} \big)_k \in \mon H_\infty(B_{\ell_r})$ (note that $\tfrac{1}{t} + \tfrac{1}{r} + \tfrac{1}{2} = 1$).   Hence $\big( k (\log(k+2))^\frac{1}{t} + \beta + \veps \big)_k \in \mon H_\infty(B_{\ell_2})$ and by our result in the case $r=2$, $\tfrac{1}{t} + \beta + \veps \geq \tfrac{1}{2}$ for every $\veps > 0$.
\end{proof}

\noindent
We are now able to give an answer to our previously stated question: the inclusion \eqref{eq:mon_epsilon} holds \emph{not} true for $\varepsilon = 0$.

\begin{corollary} \label{CORHOL1}
    Let $1 < r < \infty$ and $\tfrac{1}{s} = \tfrac{1}{2} + \tfrac{1}{\max\{r,2\}}$. Then
    \begin{equation*}
        B_{\ell_r} \cap \ell_s  \subsetneq  \mon H_\infty(B_{\ell_r}).
    \end{equation*}
\end{corollary}

\begin{proof}
    Assume equality. Let $q = \big( k \log (k+2) \big)_k$. By Theorem~\ref{thm:psigma} this implies that the diagonal operator $\ell_r \to \ell_s$ induced by the sequence $q^{-\sigma}$, where $\sigma = 1 - \tfrac{1}{\min \{ r,2 \}}$, is well-defined and by a closed graph argument bounded. Hence
    \begin{equation*}
        \bigg(
            \sum_{k=1}^\infty
            \abs[\big]{ q_k^{-\sigma} }^t
        \bigg)^\frac{1}{t}
        =   \sup_{x \in B_{\ell_p}} \Big(
            \sum_{k=1}^\infty
                \abs[\big]{ x_k \, q_k^{-\sigma} }^s
            \Big)^\frac{1}{s}
        =   \norm[]{ D_{ q^{-\sigma} } : \ell_r \to \ell_s }
        <   \infty
        \, ,
    \end{equation*}
    where $\tfrac{1}{s} = \tfrac{1}{r} + \tfrac{1}{t}$. Therefore $q^{-\sigma} \in \ell_t$. But
    \begin{equation*}
        \sum_{k=1}^\infty q_k^{-\sigma t}
        =   \sum_{k=1}^\infty
            \frac{1}{k \log (k+2)}
        =   \infty
        \, ,
    \end{equation*}
    a contradiction.
\end{proof}

\noindent
Using the same technique as in the proof of (1a) in Theorem~\ref{thm:psigma}, we easily obtain the following analog of Theorem \ref{LEMPRESQUEPOLY}.

\begin{corollary}\label{CORHOL}
    Let $1< r < \infty$ and let $\sigma = 1 - \tfrac{1}{\min\{ r, 2 \}}$. Then
	\begin{equation} \label{thm:psigma-old}
		p^{-\sigma} \cdot B_{\ell_r}
				\subset \mon H_\infty(B_{\ell_r})\,,
	\end{equation}
    and here $\sigma$ is best possible.
\end{corollary}

\begin{proof}
    We proceed analogously to the proof of (1a) in Theorem~\ref{thm:psigma} and obtain for $f\in H_\infty(B_{\ell_r})$ and $u \in B_{\ell_r}$ by Theorem~\ref{unc-basis-q},
    \begin{align*}
        \sum_{\bj} \frac{1}{p_\bj^\sigma} \abs{ c_\bj(f) u_\bj }
        &= \sum_{N = 0}^\infty \;
             \sum_{ \bj: e^N < q_\bj \le e^{N+1}}
                \frac{1}{p_\bj^\sigma} \abs{ c_\bj(f) u_\bj } \\
        &\le \sum_{N = 0}^\infty
             \frac{1}{e^{N\sigma}}\;
             \sum_{ \bj\in J(e^{N+1})}
                 \abs{ c_\bj(f) u_\bj } \\
        &\le \sum_{N = 0}^\infty
                \frac{e^{N\sigma}}{e^{(N-1)\sigma}}
                \exp \Big(
                    \big( -\sqrt2 \sigma + o(1) \big)
                    \sqrt{N \log N} \Big)
            \norm{f}_\infty
        < \infty
        \,.
        \qedhere
    \end{align*}
\end{proof}




\begin{remark}
  Analogously to the result \eqref{main} for $r=\infty$ and in view of Theorem \ref{thm:psigma}, a plausible conjecture would be that for all $r\geq 2$

  $$B_r \subset \mon H_\infty(B_{\ell_r}) \subset \overline{B}_r\,,$$
where for $\tfrac{1}{s} = \tfrac{1}{2} + \tfrac{1}{r}$
\begin{eqnarray*}
	B_r &=& \big\{
		u \in B_{\ell_ \infty}
	\,;\,
		\limsup_n \frac{1}{(\log n)^{\frac{r}{r+2}}} \sum_{k=1}^n \lvert u_k^\ast \rvert^s < 1
	\big\}\\
\overline{B}_r&=&\big\{
		u \in B_{\ell_\infty}
	\,;\,
		\limsup_n \frac{1}{(\log n)^{\frac r{r+2}}} \sum_{k=1}^n \lvert u_k^\ast \rvert^s \leq 1
	\big\}\,.
\end{eqnarray*}
\end{remark}

\begin{remark}
In Theorem \ref{thm:psigma}, the cases $1 \le r \leq 2$ and $2 \leq r \leq \infty$ do not really fit for  $r =  2$.
This is due to the fact that when we apply Theorem \ref{unc-basis-q} (2), we need that $\theta>1/2$. It would be nice to extend the statement of this last theorem to $\theta\in (0,1/2]$.
\end{remark}

\end{document}